\newif\iffull
\fulltrue

\iffull
\documentclass[a4paper,11pt]{article}
\usepackage[utf8]{inputenc}
\usepackage{amsfonts, amssymb, amsmath, amsthm}
\usepackage{mathtools}
\usepackage{enumerate}
\usepackage{graphicx}
\usepackage{authblk}
\usepackage{fullpage}
\usepackage[dvipsnames]{xcolor}
\usepackage[colorlinks=true,urlcolor=blue,citecolor=ForestGreen]{hyperref}

\newtheorem*{theorem*}{Theorem}
\newtheorem{theorem}{Theorem}

\newtheorem*{lemma*}{Lemma}
\newtheorem{lemma}[theorem]{Lemma}

\newtheorem*{proposition*}{Proposition}
\newtheorem{proposition}[theorem]{Proposition}
\newtheorem*{fact*}{Fact}

\newtheorem*{question*}{Question}
\newtheorem{question}[theorem]{Question}
\newtheorem*{corollary*}{Corollary}
\newtheorem{corollary}[theorem]{Corollary}
\newcounter{claimcounter}[theorem]
\numberwithin{claimcounter}{theorem}
\newtheorem*{claim*}{Claim}
\newtheorem{claim}[claimcounter]{Claim}

\theoremstyle{remark}
\newtheorem*{remark*}{Remark}
\newtheorem{remark}[theorem]{Remark}

\theoremstyle{definition}
\newtheorem*{definition*}{Definition}

\newtheorem*{observation*}{Observation}

\else
\documentclass[a4paper,UKenglish,cleveref, autoref]{socg-lipics-v2019}
\fi

\iffull
\title{Barycentric cuts through a convex body\footnote{The research stay of Z.P. at IST Austria is funded by the project
 Improvement of internationalization (CZ.02.2.69/0.0/0.0/17\_050/0008466) in
the field of research and development at Charles University, through the
support of quality projects MSCA-IF. The work by M.T. is supported by the GA\v{C}R grant 19-04113Y and by the Charles University projects PRIMUS/17/SCI/3 and UNCE/SCI/004.}}

\author[1,2]{Zuzana Pat\'akov\'a}
\author[3]{Martin Tancer}
\author[1]{Uli Wagner}

\affil[1]{\small IST Austria, Klosterneuburg, Austria.}
\affil[2]{\small Computer Science Institute, Charles University, Prague, Czech Republic.}
\affil[3]{\small Department of Applied Mathematics, Charles University, Prague, Czech Republic.}

\date{}

\else

\title{Barycentric cuts through a convex body}

\author{Zuzana Pat\'akov\'a}{Charles University, Prague, Czech republic, \and IST Austria, Klosterneuburg, Austria.}{zuzka@kam.mff.cuni.cz}{https://orcid.org/0000-0002-3975-1683}{The research stay at IST Austria is funded by the project
 Improvement of internationalization (CZ.02.2.69/0.0/0.0/17\_050/0008466) in
the field of research and development at Charles University, through the
support of quality projects MSCA-IF.}
\author{Martin Tancer}{Department of Applied Mathematics, Charles University, Prague, Czech
Republic.}{`my surname'@kam.mff.cuni.cz}{}{Supported by the GA\v{C}R grant 19-04113Y and by the Charles
    University projects PRIMUS/17/SCI/3 and
    UNCE/SCI/004.}
\author{Uli Wagner}{IST Austria, Klosterneuburg, Austria.}{uli@ist.ac.at}{https://orcid.org/0000-0002-1494-0568}{}

\authorrunning{Z. Pat\'akov\'a, M. Tancer, U. Wagner}
\Copyright{Zuzana Patáková, Martin Tancer, Uli Wagner}



\acknowledgements{We thank Stanislav Nagy for introducing us to Gr\"{u}nbaum's questions,
for useful discussions on the topic, for providing us with many references, and
for comments on a preliminary version of this paper. We thank Jan Kyn\v{c}l and
Pavel Valtr for letting us know about a more general counterexample they found, and
Roman Karasev for pointing us to related work~\cite{karasev11,blagojevic-karasev16} and for comments on a
preliminary version of this paper.
Finally, we thank an anonymous referee for
many comments on a preliminary version of the paper which, in particular,
yielded an important correction in Section~\ref{s:one_more_point}.
}

\EventEditors{Sergio Cabello and Danny Z. Chen}
\EventNoEds{2}
\EventLongTitle{36th International Symposium on Computational Geometry (SoCG 2020)}
\EventShortTitle{SoCG 2020}
\EventAcronym{SoCG}
\EventYear{2020}
\EventDate{June 22--26, 2020}
\EventLocation{Z\"{u}rich, Switzerland}
\EventLogo{socg-logo}
\SeriesVolume{164}
\ArticleNo{62}

\theoremstyle{theorem}
\newtheorem{question}[theorem]{Question}
\fi

\newtheorem{postulate}[theorem]{Postulate}
\newtheorem{embeddedclaim}{Claim}[theorem]

\usepackage{tikz-cd}
\newcommand{\R}{\mathbb{R}}

\newcommand{\RP}{\mathbb{R}P}
\renewcommand{\O}{\mathcal{O}}
\renewcommand{\P}{\mathbf{P}}

\DeclareMathOperator{\dep}{depth}
\DeclareMathOperator{\conv}{conv}
\DeclareMathOperator{\cen}{cen}

\newcommand{\marrow}{\marginpar{\boldmath$\longleftarrow$}}
\newif\ifcmts

\cmtstrue

\ifcmts
\newcommand{\martin}[1]{\ifhmode\newline\fi\marrow
  \textsf{\textcolor{green}{\bf
MARTIN:} #1\newline}}
\newcommand{\zuzka}[1]{\ifhmode\newline\fi\marrow
  \textsf{\textcolor{magenta}{\bf
ZUZKA:} #1\newline}}
\newcommand{\uli}[1]{\ifhmode\newline\fi\marrow
  \textsf{\textcolor{cyan}{\bf
ULI:} #1\newline}}
\else
\newcommand{\martin}[1]{}
\newcommand{\zuzka}[1]{}
\newcommand{\uli}[1]{}
\fi

\newcommand{\blue}[1]{\textcolor{blue}{#1}}

\begin{document}

\maketitle

\begin{abstract}
Let $K$ be a convex body in $\R^n$ (i.e., a compact convex set with nonempty
  interior). Given a point $p$ in the interior of $K$, a hyperplane $h$ passing
  through $p$ is called \emph{barycentric} if $p$ is the barycenter of $K \cap
  h$. In 1961, Gr\"{u}nbaum raised the question whether, for every $K$, there
  exists an interior point $p$ through which there are at least $n+1$ distinct
  barycentric hyperplanes. 
  Two years later, this was seemingly resolved
  affirmatively by showing that this is the case if $p=p_0$ is
  the point of maximal \emph{depth} in $K$. However, while working on a related
  question, we noticed that one of the auxiliary claims in 
 the proof
  is incorrect. Here, we provide a counterexample; this re-opens Gr\"unbaum's
  question.

It follows from known results that for $n \geq 2$, there are always at least
  three distinct barycentric cuts through the point $p_0 \in K$ of maximal
  depth. Using 
  tools related to  Morse theory
  we are able to improve this bound: four distinct barycentric cuts through $p_0$ are guaranteed if $n \geq 3$. 
\end{abstract}

\section{Introduction}
\subparagraph*{Gr\"{u}nbaum's questions.} Let $K$ be a convex body in $\R^n$ (i.e., compact convex set with nonempty interior). Given an interior point $p\in K$, a hyperplane $h$ passing through $p$ is called \emph{barycentric} if $p$ is the barycenter (also known as the centroid) of the intersection $K\cap h$. In 1961, Gr\"{u}nbaum~\cite{grunbaum61} raised the following questions (see also \cite[\S 6.1.4]{grunbaum63}):
\begin{question}
\label{q:1}
Does there always exist an interior point $p\in K$ through which there are at least $n+1$ distinct barycentric hyperplanes?  
\end{question}
\begin{question}
\label{q:2}
In particular, is this true if $p$ is the barycenter of $K$?
\end{question}
Seemingly, Question~\ref{q:1} was  answered affirmatively by Gr\"unbaum himself~\cite[\S
6.2]{grunbaum63} two years later, by using a variant of Helly's theorem to show that there are at least $n+1$ barycentric cuts through the point of $K$ of maximal \emph{depth} (we will recall the definition below).
The assertion that Question~\ref{q:1} is resolved has also been reiterated in other geometric literature~\cite[A8]{croft_falconer_guy94}. However, when working on Question~\ref{q:2}, which remains open, we identified a concrete problem in Grünbaum's argument for the affirmative answer for the point of the maximal depth. 
The first aim of this paper is to point out this problem, which re-opens Question~\ref{q:1}. 
\subparagraph*{Depth, depth-realizing hyperplanes, and the point of maximum
depth.} In order to describe the problem with Gr\"unbaum's argument, we need a
few definitions. Let $p$ be a point in $K$. For a unit vector $v$ in the unit
sphere $S^{n-1} \subseteq \R^n$,  let $h_v = h_v^p:=\{x\in \R^n \colon \langle
v,x-p\rangle =0\}$ be the hyperplane orthogonal to $v$ and passing through $p$, and let $H_v = H_v^p:= \{x\in \R^n \colon \langle v,x-p\rangle  \geq 0\}$ be the half-space bounded by $h_v$ in the direction of $v$. Given $p$, we define the \emph{depth function} $\delta^p\colon
S^{n-1} \to [0,1]$ via $\delta^p(v) = \lambda(H_v \cap
K)/\lambda(K)$, where $\lambda$ is the Lebesgue measure ($n$-dimensional volume) in $\R^n$.
The \emph{depth} of a point $p$ in $K$ is defined as 
$\dep(p, K) := \inf_{v \in S^{n-1}} \delta^p(v).$ It is easy to
see\footnote{Given $v, v' \in S^{n-1}$, $\lambda(H_v \cap
K)$ and $\lambda(H_{v'} \cap K)$ differ by at most $\lambda((H_v \Delta H_{v'})
\cap K)$ where $\Delta$ is the symmetric difference. For $\varepsilon > 0$
and $v$ and $v'$ sufficiently close, $\lambda((H_v \Delta H_{v'})
\cap K) < \varepsilon \lambda(K)$~as~$K$~is~bounded.}
that
$\delta^p$ is a continuous function, therefore the infimum in the definition is
attained at some $v \in S^{n-1}$. Any hyperplane $h_v$ through $p$ such that $\dep(p, K) =
\delta^p(v)$ is said to \emph{realize the depth} of $p$. Finally, a \emph{point
of maximal depth in $K$} is a point $p_0$ in the interior of $K$ such that
$\dep(p_0, K) := \max \dep(p, K)$ where the maximum is taken over all points in
the interior of $K$.\footnote{We remark that our depth function slightly differs from the
function $f(H,p)$ used by Gr\"{u}nbaum~\cite[\S 6.2]{grunbaum63}. However, the
point of maximal depth coincides with the `critical point' in~\cite{grunbaum63} and 
hyperplanes realizing the depth for $p_0$ coincide with the `hyperplanes through
the critical point dividing the volume of $K$ in the ratio $F_2(K)$'.}
The point of maximal depth always exists (by compactness of $S^{n-1}$) and
it is unique (two such points would yield a point of larger depth on the segment between them).

\subparagraph{Many depth-realizing hyperplanes?} Gr\"unbaum's argument has two ingredients. The first is the following result, known as Dupin's theorem~\cite{dupin22}, which dates back to 1822:
\begin{theorem}[Dupin's Theorem]
\label{thm:Dupin}
If a hyperplane $h$ through $p$ realizes the depth of $p$ then it is barycentric with respect to $p$.
\end{theorem}
Gr\"unbaum refers to Blaschke~\cite{blaschke17} for a proof; for a more recent reference, see~\cite[Lemma 2]{schutt-werner94}.\footnote{The idea of the proof is simple: For contradiction assume that $h$ realizes the depth of
$p$ but that the barycenter $b$ of $K \cap h$ differs from
$p$. Let $v \in S^{n-1}$ be such that $h = h_v$ and $\dep(p, K) = \delta^p(v)$.
Consider the affine $(d-2)$-space $\rho$ in $h$ passing through $p$ and
perpendicular to the segment $bp$. Then by a small rotation of $h$ along $\rho$
we can get $h_{v'}$ such that $\delta^p(v') < \delta^p(v)$ which contradicts
that $h$ realizes the depth of $p$.
Of course, it remains to check the
details.}
A stronger statement will be the content of
Proposition~\ref{p:C1} below.

The second ingredient in Gr\"unbaum's argument is the following assertion (which in~\cite[\S 6.2]{grunbaum63} is deduced using a variant of Helly's theorem, without providing the details).
\begin{postulate}
\label{postulate}
If $p_0$ is the point of $K$ of maxiumal depth, then there are at least $n+1$ distinct hyperplanes through $p_0$ that realize the depth.
\end{postulate}
If correct, Postulate~\ref{postulate}, in combination with Dupin's theorem, would immediately imply an affirmative answer to Question~\ref{q:1}. 
However, it turns out that this step is problematic. Indeed, there is a counterexample to Postulate~\ref{postulate}:

\begin{proposition}
\label{p:counterexample}
Let $K = T \times I \subseteq \R^3$ where $T$ is an 
equilateral triangle and $I$ is a line segment (interval) orthogonal to $T$, and let $p_0 \in K$ be the point of maximal depth  (which in this case coincides with the barycenter of $K$). Then there are only $3$ hyperplanes realizing the depth of $p_0$. 
 
\end{proposition}
\begin{remark} We believe that Proposition~\ref{p:counterexample} can be
  generalized to higher dimensions in the sense that, for every $n$, there are
  only $n$ depth-realizing hyperplanes through the point of maximal depth in
  $\Delta \times I \subseteq \R^n$, where $\Delta$ is a regular
  $(n-1)$-simplex. However, we did not attempt to work out the details
  carefully, because Kyn\v{c}l and Valtr~\cite{kyncl-valtr19} informed us about
  stronger counterexamples: For every $n$, there exists a convex body $K \in \R^n$ such that there are only $3$ depth-realizing hyperplanes through the point of maximal depth in $K$. Therefore, we prefer to keep the proof of Proposition~\ref{p:counterexample} as
simple as possible and focus on dimension $3$.
\end{remark}


\begin{remark} We emphasize that Proposition~\ref{p:counterexample} does not preclude an affirmative answer to Gr\"unbaum's Question~\ref{q:1} (nor to Question~\ref{q:2}), since $T \times I$ contains infinitely many distinct barycentric hyperplanes through $p_0$.
Thus Gr\"{u}nbaum's questions remain open. 
\end{remark}

We also remark that a weakening of Postulate~\ref{postulate} is known to be true (see the `Inverse Ray Basis Theorem~\cite{rousseeuw-ruts99}, using the proof from~\cite{donoho-gasko92}):\footnote{We remark that the second condition in the statement of the result in~\cite{rousseeuw-ruts99} is equivalent to the statement that $0 \in \conv U$, in our notation.}$^{,}$\footnote{Sketch of the inverse ray basis theorem:
if there is a closed hemisphere $C \subseteq S^{n-1}$ which does not contain
a point of $U$, let $v$ be the centre of $C$. Then a small shift of $p_0$ in
the direction of $v$ yields a point of larger depth, a contradiction.}

\begin{proposition}
\label{p:irbt}
Let $U \subseteq
S^{n-1}$ be the set of vectors $u$ such that $\delta^{p_0}(u) = \dep(p_0, K)$. Then $0 \in \conv U$.
\end{proposition}

In the special case that $U$ is in general position, the cardinality of $U$ is
at least $n+1$ (otherwise $\dim \conv U < n$ and $\conv U$ would not contain
the origin, by general position), which proves Postulate~\ref{postulate} in
this special case. However, $U$ need not be always in general position. For
example, in the case $K = T \times I$ in $\R^3 = \R^2 \times \R$ of
Proposition~\ref{p:counterexample}, the set $U$ contains three vectors in the
plane through the origin parallel with $T$. This is also the way we arrived at
the counterexample from Proposition~\ref{p:counterexample}.



Inverse Ray Basis Theorem immediately implies that three barycentric
hyperplanes are guaranteed in dimension at least $2$.

\begin{corollary}
  \label{c:three}
Let $K$ be a convex body in $\R^n$ where $n \geq 2$ and $p_0$ be the point of
  maximal depth of $K$. Then there at least three distinct barycentric
  hyperplanes through $p_0$.
%
\end{corollary}

\begin{proof}
  Let $U$ be the set from Proposition~\ref{p:irbt}.  Then, $0 \in \conv U$ and $U \subseteq S^{n-1}$ imply
together $|U| \geq 2$. However, if $|U| = 2$, then $U = \{u,-u\}$ for some $u \in
S^{n-1}$. This necessarily means $\dep(p_0, K) = \delta^{p_0}(u) =
\delta^{p_0}(-u) = 1/2$ as $\delta^{p_0}(u) + \delta^{p_0}(-u) = 1$. Then for any other $v \in S^{n-1}$ we get
$\min\{\delta^{p_0}(v), \delta^{p_0}(-v)\} \geq 1/2$ which implies $\delta^{p_0}(v) =
\delta^{p_0}(-v) = 1/2$ as well. Therefore $v \in U$ contradicting $|U| = 2$.)
\end{proof}

\subparagraph{Four barycentric cuts via critical points of $C^1$
functions.} Using tools related to 
 Morse theory, we are able to obtain one more
barycentric hyperplane, provided that $n \geq 3$.

\begin{theorem}
\label{t:four}
Let $K$ be a convex body in $\R^n$ where $n \geq 3$ and $p_0$ be the point of
  maximal depth of $K$. Then there are at least four distinct hyperplanes $h$ such that $p_0$ is the barycenter of $K \cap h$. 
\end{theorem}

Here we should also mention related work of Blagojevi\'{c} and
Karasev~\cite[Theorem~3.3]{karasev11} and
\cite[Theorem~1.13]{blagojevic-karasev16}. They show that there are at least
$\mu(n)$ barycentric hyperplanes passing through \emph{some} interior point of
$K$ (not necessarily the point of maximal depth), where $\mu(n):=\min_f
\max_{p\in S^n} |f^{-1}(p)|$ is the minimum \emph{multiplicity} of any
continuous map $f\colon \RP^n\to S^n$ (here, $\R P^n$ is the $n$-dimensional
real projective space).  By calculations with Stiefel--Whitney classes, they
obtain lower bounds for $\mu(n)$ that depend in a subtle (and non-monotone) way
on $n$ (see \cite[Remark~1.3]{karasev11}). For example, $\mu(n)\geq
\frac{n}{2}+1$ if $n=2^\ell -2$, but for values of $n$ of the form $n=2^\ell-1$
(e.g., for $n=3$) their methods only give a lower bound of $\mu(n)\geq 2$.

Our argument in the proof of Theorem~\ref{t:four} is, in certain sense,
tight. For completeness we discuss this in Section~\ref{s:conclusions}.

In what follows, 
we view 
$S^{n-1}$ as a smooth manifold with its 
standard differential structure. A key tool in the proof of Theorem~\ref{t:four} is the following close connection between barycentric hyperplanes and the critical points of the depth function:

\begin{proposition}
\label{p:C1}
  Let $K \subseteq \R^n$ be a convex body and $p$ be a point in the interior of
  $K$. Then the corresponding depth function $\delta^p \colon S^{n-1} \to \R$
  is a $C^1$ function. In addition, $v \in S^{n-1}$ is a critical point of
  $\delta^p$ (that is, $D\delta^p(v) = 0$, where $Df(v)$ denotes the total derivative of a function $f$
  at $v$) if and only if
  $h_v$ is barycentric. 
\end{proposition}

As mentioned earlier, Proposition~\ref{p:C1} generalizes Dupin's
theorem.
Indeed, if $h = h_v$
realizes the depth, then $v$ is a global minimum of $\delta^p$, hence $h$ is barycentric by Proposition~\ref{p:C1}. 

In the proof, we closely follow computations by Hassairi and
Regaieg~\cite{hassairi-regaieg08} who stated an extension of Dupin's
theorem to absolutely continuous probability measures.
As explained in~\cite{nagy-schutt-werner19} (see Proposition~29, Example~7, and
the surrounding text in~\cite{nagy-schutt-werner19}), the extension of Dupin's
theorem
does not hold in the full generality stated
in~\cite{hassairi-regaieg08}, and it requires some additional assumptions.
However, a careful check of the computations of Hassairi and
Regiaeg~\cite{hassairi-regaieg08} in the special case of uniform probability
measures on convex bodies reveals not only Dupin's theorem but all items of
Proposition~\ref{p:C1}.





Regarding the proof of Theorem~\ref{t:four}, the 
Inverse Ray Basis Theorem (Proposition~\ref{p:irbt}) and Corollary~\ref{c:three}
imply that $\delta^{p_0}$ has at least three global minima. This
gives three barycentric hyperplanes via Proposition~\ref{p:C1}.
Furthermore, we also get three maxima of $\delta$, as a maximum
appears at $v$, if and only if a minimum appears at $-v$ (note that $h_v =
h_{-v}$). However, it
should not happen for a $C^1$ function on $S^{n-1}$ that it has only such critical
points. We will show that there is at least one more critical point, which
yields another barycentric hyperplane via Proposition~\ref{p:C1}. Namely, we
show the following proposition.


\begin{proposition}
 \label{p:one_more_point}
Let $n \geq 2$ and let $f \colon S^n \to \R$ be a $C^1$ function. Let $m_1,
  \dots, m_k$ be (not necessarily strict) local minima or maxima of $f$, where $k \geq 3$. Then there exists
  $u \in S^n$, different from $m_1, \dots, m_k$, such that $Df(u) = 0$. 
\end{proposition}


This finishes the proof of Theorem~\ref{t:four} modulo
Propositions~\ref{p:C1} and~\ref{p:one_more_point}.
(Proposition~\ref{p:one_more_point} is applied with $k = 6$.)

The main idea beyond the proof of Proposition~\ref{p:one_more_point} is
that if we have at least three local minima or maxima, then we should also
expect a saddle point (unless there are infinitely many local extremes).
This would be
an easy exercise for Morse functions (which are in particular $C^2$) via Morse
theory (actually, the Morse inequalities would provide even more critical points).
Working with $C^1$ functions adds a few difficulties, but all of
them can be overcome.

\subparagraph*{Relation to probability and statistics.} 
The depth function, as we define it above is a special case of the (Tukey) depth of a probability
measure in $\R^d$, a well-known notion in statistics~\cite{tukey75,donoho82,donoho-gasko92}. More precisely, given a probability measure $\P$ on $\R^d$
and $p \in \R^d$, we can define $\dep(p, \P) := \inf_{v \in S^{n-1}}\P(H_v)$.
Then $\dep(p, K)$ is a special case of the uniform probablity measure on a convex body $K$, i.e., 
$\P(A) := \lambda(A)/\lambda(K)$ for $A$ Lebesgue-measurable. 
We refer
to~\cite{nagy-schutt-werner19} for an extensive recent survey making
many connections between the depth function in statistics and geometric
questions.

There is a vast amount of literature, both in computational geometry and statistics, devoted
to computing the depth function in various settings (which is not easy in
general). We refer, for example,
to~\cite{rousseeuw-struyf98, chan04, bremner-chen-iacono-langerman-morin08,chen-morin-wanger13,
dyckerhoff_mozharovskyi16, liu_mosler_mozharovskyi19} and the references
therein. From this point of view, understanding the minimal possible number of
critical points of the depth function is a quite fundamental property of the
depth function. Via Proposition~\ref{p:C1}, this is essentially equivalent to Gr\"{u}nbaum's questions.


%

\subparagraph*{Organization.} Proposition~\ref{p:counterexample} is proved in
Section~\ref{s:counterexample}; Proposition~\ref{p:C1} is proved in
Section~\ref{s:C1}; and Proposition~\ref{p:one_more_point} is proved in
Section~\ref{s:one_more_point}.

\section{Few hyperplanes realizing the depth}
\label{s:counterexample}

In this section we prove Proposition~\ref{p:counterexample}, assuming Proposition \ref{p:C1}.

\subparagraph*{Preliminaries.} Let us recall that given a bounded measurable set $Y \subseteq \R^n$ of
positive measure, the
\emph{barycenter} of $Y$ is defined as 
\begin{equation}
\label{e:barycenter}
\cen Y = \frac{\int_{\R^n} x \chi_Y(x) dx}{\int_{\R^n}
\chi_Y(x) dx} = \frac{1}{\lambda(Y)}\int_{Y} x dx
\end{equation}
where $\chi_Y$ is the characteristic function and the integral is
considered as a vector in $\R^n$. 

If $Y$ splits as a disjoint union $Y
= Y_1 \sqcup \cdots \sqcup Y_\ell$ of sets of positive measure then 
\begin{equation}
  \label{e:sum_centres}
  \cen Y = \frac1{\lambda(Y)}\left(\sum_{i=1}^{\ell}\lambda(Y_i) \cen
  Y_i\right)
\end{equation}
 which easily follows from~\eqref{e:barycenter}.

If $h$ is a hyperplane, and $Y \subseteq h$ has positive $(n-1)$-dimensional
Lebesgue measure inside $h$, then the formula for the barycenter is analogous
to~\eqref{e:barycenter}:
\begin{equation}
\label{e:barycenter_h}
  \cen Y = \frac{\int_{h} x \chi_Y(x) d\lambda_{n-1}(x)}{\int_h
  \chi_Y(x) d\lambda_{n-1}(x)} = \frac{1}{\lambda_{n-1}(Y)}\int_{Y} x
  d\lambda_{n-1}(x)
\end{equation}
where, for purpose of this formula, $\lambda_{n-1}$ denotes the
$(n-1)$-dimensional Lebesgue measure on $h$. 

If $h \subseteq \R^n$ is a hyperplane whose 
orthogonal projection $\pi(h)$ onto $\R^{n-1}\times\{0\}$ 
(the first $n-1$ coordinates) 
equals $\R^{n-1}\times\{0\}$, 
then $\cen \pi(Y) = \pi(\cen Y)$.

\begin{proof}[Proof of Proposition~\ref{p:counterexample}.]
Let 
$T \subseteq \R^2$ be 
an equilateral triangle with 
  $\cen(T) = 0$ and $I = [-1,1]$. Then $\cen(K) = 0$. In addition, because the
  point of maximal depth $p_0$ is unique and invariant under isometries of $K$,
  we get $p_0  = 0$.

We will use the following notation: $a$, $b$, $c$ are the vertices of $T$ and
  $\alpha$, $\beta$, and $\gamma$ are lines perpendicular to $T$ passing through
  $a$, $b$, and $c$ respectively.

Now let $h$ be a hyperplane passing through $0$. We want to find out whether
  $h$ realizes the depth. We will consider three cases: 
  
  \begin{enumerate}[(i)]
   \item $h$ is perpendicular to $T$; \label{c:i}
   \item $h$ is not perpendicular to $T$ and all
  intersection points of $h$ with $\alpha$, $\beta$, and $\gamma$ belong to~$K$; \label{c:ii}
  \item $h$ is not perpendicular to $T$ and at least one
  of the  intersection points of $h$ with $\alpha$, $\beta$, and $\gamma$ does \label{c:iii}
  not belong to $K$.
  \end{enumerate}

  In case (\ref{c:i}), we will find three candidates for hyperplanes
  realizing the depth. Then we show that there is no hyperplane realizing the
  depth in cases (\ref{c:ii}) and (\ref{c:iii}), which shows that only the three
  candidates from case (\ref{c:i}) may realize the depth. They
  realize the depth because we have at least three hyperplanes realizing the
  depth by the discussion in the introduction above
  Theorem~\ref{t:four}.
  

  Let us focus on case (\ref{c:i}). 
This is the same as considering the lines realizing the depth in an
 equilateral triangle.
  It is easy to check and well known (see
  e.g.~\cite[\S 5.3]{rousseeuw-ruts99}) that the depth of the equilateral
  triangle 
  is $4/9$ and it is realized by lines parallel with the sides of the
  triangle. It follows that we can reach depth $4/9$ in $K$ by hyperplanes perpendicular to
  $T$ and parallel with the three sides of $T$, and all other hyperplanes from
  case (i) bound a portion of $K$ strictly larger than $4/9$ on each of their sides.

 Case (\ref{c:ii}) is very easy: It is easy to compute that each hyperplane of type 
  (\ref{c:ii}) splits $K$ into two parts of equal volume $1/2$.
  Therefore, no such hyperplane 
  realizes the depth.

 Finally, we investigate case (\ref{c:iii}). Here we show that no hyperplane $h$
  of case (\ref{c:iii}) is barycentric. Therefore, by Theorem~\ref{thm:Dupin},
  it cannot realize the depth either.

  We aim to show that $0$ is not the barycenter of $h \cap K$. Let $U$ be the
  orthogonal projection of $h \cap K$ to the triangle $T$. Equivalently, we want to show
  that $0$ is not the barycenter of $U$. We also realize that $U = T \cap S$,
  where $S$ is an infinite strip obtained as the orthogonal projection of $h
  \cap (\R^2 \times I)$ to $\R^2\times \{0\}$
  ; see Figure~\ref{f:U}.

Let $s$ be the center line of $S$. This is the line where $h$ meets the plane
of $T$. 
  We remark that $0$ belongs to $s$ and in addition $U$ is a proper subset of $T$ (otherwise we would
  be in case (\ref{c:ii})). We again distinguish three cases: 
  
  \begin{enumerate}[(a)]
   \item none of the vertices $a, b, c$ belongs to $U$,
   \item  one of the vertices $a, b, c$ belongs to $U$,
   \item two of the vertices $a, b, c$ belong to $U$.
  \end{enumerate}

  In all the cases we will show $\cen U \neq \cen T$.
  \begin{figure}
\begin{center}
 \includegraphics[page=2]{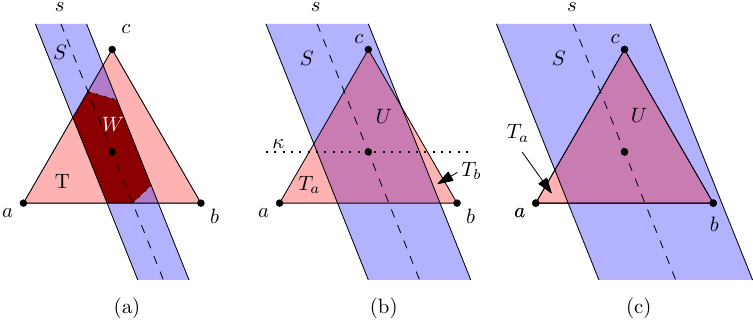}
\caption{Three cases for the intersection $U = T \cap S$.}
  \label{f:U}
\end{center}
  \end{figure}
%
  In case (a), $s$ splits one of the vertices of $T$ from the other two.
  Without loss of generality, $a$ is on one side of $s$ and $b$ and $c$ are on
  the other side. The center line $s$ also splits $U$ into two parts. Let $W'$
  be the (closed) part on the side of $a$, $W''$ be the mirror image of $W'$ along $S$
  and $W := W' \cup W''$. Note that $W$ is a proper subset of $U$; indeed, since $\cen T = 0$ and $T$ is equilateral, the line $s$ splits the segment $ab$ closer to $b$ and the
  segment $ac$ closer to $c$. By the symmetry of $W$, the barycenter $\cen W$
  belongs to the line $s$. However, this means that the barycenter of $U$ is
  not on $s$; it is on the $bc$ side of $s$. Formally, this follows
  from~\eqref{e:sum_centres} for the decomposition $U = W \sqcup (U \setminus W)$. 

In case (b), without loss of generality, $U$ contains $c$. Then $T \setminus
U$ is the union of two triangles $T_a$ and $T_b$. Let $\kappa$ be the line
parallel with $ab$ passing through $0$. Without loss of generality, up to
rotating $T$, $\kappa$ is the $x$-axis. From~\eqref{e:sum_centres}, we get
$0 = \cen T = \frac1{\lambda(T)}(\lambda(U) \cen U + \lambda(T_a) \cen T_a +
\lambda(T_b) \cen T_b)$. The barycenters $\cen T_a$ and $\cen T_b$ are below
the line $\kappa$ or on it. At least one of these barycenters is strictly
below ($\cen T_a$ is on $\kappa$ if and only if $c$ belongs to the closure of
$T_a$, and similarly with $T_b$). 
Therefore, $\cen U$ must be strictly above
$\kappa$ if the above equality is supposed to hold. 

In case (c), it is even more obvious that $\cen U \neq \cen T$. Without loss of
generality $U$ contains $b$ and $c$. Then $T \setminus U$ is a triangle $T_a$.
Since both $T$ and $T_a$ are convex and $T_a$ does not contain $\cen T$, we have $\cen T_a \neq \cen T$. Therefore $\cen T
\neq \cen U$ follows from~\eqref{e:sum_centres} for the decomposition $T = U
\sqcup T_a$.
\end{proof}

\section{Critical points of the depth function}
\label{s:C1}
Here we prove Proposition~\ref{p:C1}. We follow~\cite{hassairi-regaieg08}
with a slightly adjusted notation and adding a few more details here and there.

\begin{proof}[Proof of Proposition~\ref{p:C1}]
Without loss of generality, we can assume that the point $p$ coincides
with the origin and we suppress it from the notation. That is, we write
$\delta$ for the depth function instead of $\delta^p$.


Let $e_1, \ldots, e_n$ be the canonical basis of $\R^n$ and let
\[ S_{j+}^{n-1} = \{u = \sum_{i=1}^n u_ie_i \in S^{n-1}; u_j >
0\} \hskip5mm \hbox{ and } \hskip5mm S_{j-}^{n-1} = \{u = \sum_{i=1}^n u_ie_i \in S^{n-1}; u_j <
0\}
\]
be the relatively open hemispheres  of $S^{n-1}$ with poles at $e_j$ and
$-e_j$, for $j \in [n]$. 
These sets
form an atlas on $S^{n-1}$.

Let us consider $j \in [n]$.
Given $x \in \R^n$ and $i \in [n]$, $x_i$ denotes the $i$th coordinate of
$x$, that is $x= \sum_{i=1}^n x_ie_i$.  With a slight abuse of the notation, we
identify $\R^{n-1}$ with the subspace of $\R^n$ spanned by $e_1, \dots,
e_{j-1}, e_{j+1}, \dots, e_n$. Let $\hat x := \sum_{i=1, i \neq j}^n x_ie_i
\in \R^{n-1}$.
%
 Following~\cite{hassairi-regaieg08} we consider the diffeomorphisms $u \mapsto
 \beta(u) = -\frac{\hat u}{u_j}$ between $S_{j+}^{n-1}$ and $\R^{n-1}$ or
  between $S_{j-}^{n-1}$ and $\R^{n-1}$. 
We will check the required properties of $\delta$ locally at each
of the $2n$ hemispheres $S_{j+}^{n-1}$ or $S_{j-}^{n-1}$ (with respect to the aforementioned
diffeomorphisms). Given that all cases are symmetric, it is sufficient to focus
only on the $S_{n+}^{n-1}$ case. 
  That is, from now on, we assume that $j = n$ and $\R^{n-1}$ is spanned
  by the first $(n-1)$ coordinates in the convention above. Given a point $x
  \in \R^n$, we also write it as $x = ( \hat x; x_n)$.



 Now, for $y \in \R^{n-1}$ we consider the hyperplane $h'_y$ in $\R^n$ containing the origin and defined by 
 \[
  h'_y = \{(\hat x; x_n) \in \R^n \colon x_n = \langle y,\hat x\rangle\}.
 \]
Note that if $u \in S^{n-1}_{j+}$, then $h'_{\beta(u)} = \{x \in
\R^n\colon \langle x, u \rangle = 0\}$.
In particular, since $p$ is the origin, $h'_{\beta(u)}$ coincides with $h_u$ used in
the introduction for definition of the depth function. This also means
that the map
$y \mapsto h'_y$ provides a parametrization of a family of those
hyperplanes containing the origin which do not contain $e_n$. We also
set $H'_y$ to be the positive halfspace bounded by $h'_y$:
\[
 H'_y = \{(\hat x; x_n) \in \R^n \colon x_n \geq \langle y,\hat x\rangle \}.
\]
Again, if $u \in S^{n-1}_{j+}$, then $H'_{\beta(u)}$ coincides with $H_u$ from
the introduction (here we use $u_n > 0$).


Now, we consider the map $f \colon \R^{n-1} \to \R$ defined by 
\begin{equation}
\label{e:big_delta}
 f(y) = \lambda (H'_y \cap K) = \int_{\R^{n-1}} \int_{\langle y,
 \hat x \rangle}^{\infty} \chi_K(\hat x; x_n) dx_n d\hat x, 
\end{equation}
where $\chi_K$ is the characteristic function of $K$. 
When $y = \beta(u)$ for some $u \in S^{n-1}_{j+}$, then $f(\beta(u)) =
\delta(u)$. Therefore, given that the map $u \to \beta(u)$ is a diffeomorphism,
it is sufficient to prove that $f$ is a $C^1$ function and that
$\beta(v) \in \R^{n-1}$ is a critical point of $f$ if and only if
$h'_{\beta(v)} = h_v$ is barycentric.



The aim now is to differentiate $f(y)$ with respect to $y$.
We will show that the total derivative equals
\begin{equation}
\label{e:df}
Df(y) = - \int_{\R^{n-1}} \hat x \cdot \chi_K(\hat x; \langle y, \hat
x\rangle) d \hat x
\end{equation}
considering the integral on the right-hand side as a vector.
Deducing~\eqref{e:df} is a quite routine computation skipped
in~\cite{hassairi-regaieg08}.\footnote{When compared with formula (3.1)
in~\cite{hassairi-regaieg08}, we obtain a different sign in front of the
integral. This is caused by integration over the opposite halfspace.} 
However, this is the step in the proof of Theorem~3.1 
in~\cite{hassairi-regaieg08} which reveals that some extra assumptions
in~\cite{hassairi-regaieg08} are necessary. Thus we carefully
deduce~\eqref{e:df} at the end of this proof for completeness.


We will also see that all partial derivatives of $f$ are continuous
which
means that $f$ is a $C^1$ function
which is one of our required conditions. Now we want to show that
$Df(\beta(v)) = 0$ if and only if $h_v$ is barycentric.

First, assume that $Df(\beta(v)) = 0$.
This gives
\begin{equation}
\label{e:zero_barycenter}
  0 = \frac{\int_{\R^{n-1}} \hat x \cdot \chi_K(\hat x; \langle \beta(v), \hat
  x\rangle) d \hat x}{\int_{\R^{n-1}} \chi_K(\hat x; \langle
  \beta(v), \hat
    x\rangle) d \hat x}
\end{equation}
which means that $0$ is the barycenter of $K \cap h'_{\beta(v)}$ from the
definition of $h'_{\beta(v)}$. 

On the other hand, if $0$ is the barycenter of $K \cap h'_{\beta(v)}$,
then we deduce~\eqref{e:zero_barycenter} which implies $Df(\beta(v)) = 0$.


It remains to show~\eqref{e:df}. For this purpose, we compute partial
derivatives $\frac{\partial}{\partial y_k} f(y)$, $1 \leq k \leq n-1$. 
In the following computations, recall that $e_k$ stands for the standard basis vector
for the $k$th coordinate and let $\int_a^b := -\int_b^a$ if $a > b$. We get
\begin{align*}
 \frac{\partial}{\partial y_k}f(y) 
  &=  \lim_{t \to 0}\frac1{t}\int_{\R^{n-1}}\left(\int_{\langle y + t e_k, \hat
  x\rangle}^{\infty} \chi_K(\hat x; x_n) d x_n - \int_{\langle y, \hat
    x\rangle}^{\infty} \chi_K(\hat x; x_n) d x_n\right) d\hat x \\
    &=  \lim_{t \to
    0}\int_{\R^{n-1}}\frac1t\int_{\langle y, \hat
      x\rangle + t x_k}^{\langle y, \hat x\rangle} \chi_K(\hat x; x_n)
      d x_n d\hat x.
\label{e:partial_fy}
\end{align*}

Let $y, \hat x \in \R^{d-1}$ be such that $( \hat x; \langle y, \hat x\rangle)
\not \in \partial K$. Then we get
\[
  \lim_{t\to 0}  \frac1t\int_{\langle y, \hat
        x\rangle + t x_k}^{\langle y, \hat x\rangle} \chi_K(\hat x; x_n)
	      d x_n = -x_k\chi_K(\hat x; \langle y, \hat x \rangle),
\label{e:limit_inner_integral}
\]
 because $( \hat x; \langle y, \hat x\rangle)
 \not \in \partial K$ implies that the function $\chi_K(\hat x; x_n)$ as a
 function of $x_n$ is constant on the interval $(\langle y, \hat
 x\rangle - |t x_k|, \langle y, \hat  x\rangle + |t x_k|)$ for small enough
 $|t|$. Therefore, by the dominated convergence theorem, 
\begin{equation}
 \frac{\partial}{\partial y_k}f(y) = \int_{\R^{n-1}}-x_k \chi_K(\hat x; \langle
 y, \hat x\rangle) d\hat x. 
\label{e:partial_fy_final}
\end{equation}
 
 For fixed $y$, the condition $( \hat x; \langle y, \hat x\rangle)
  \not \in \partial K$ holds for almost every $\hat x$ because $( \hat x;
  \langle y, \hat x\rangle) \in h_y$ and $h_y$ passes through the interior of
  $K$ (through the origin). 
%
  By another application of dominated convergence
  theorem, we realize that the right hand side of~\eqref{e:partial_fy_final} is
  continuous in $y$ (this time, we consider a sequence $y^i \to y$ and we
  observe that $\chi_K(\hat x; \langle y^i, \hat x\rangle) \to \chi_K(\hat
  x; \langle y, \hat x\rangle)$ for almost every $\hat x$). Therefore the total
  derivative of $f$ at any $y$ exists and \eqref{e:partial_fy_final} gives the
  formula~\eqref{e:df}.
\end{proof}

\begin{remark}
  In the last paragraph of the proof above we crucially use the convexity of $K$.
  Without convexity, there is a compact nonconvex polygon $K' \subseteq \R^2$,
  with $0$ in the interior, such that there is $y$ with the property that the
  set of those $\hat x$ for which $( \hat x; \langle y, \hat x\rangle)
   \in \partial K'$ has positive measure;
  see Figure~\ref{f:two_squares}.
  In fact, even~\eqref{e:df} does not hold for $K'$. Here we took $K'$ to be
  the polygon from Example~7 of~\cite{nagy-schutt-werner19}, and we refer the reader to that paper for more
  details.
\end{remark}

\begin{figure}
\begin{center}
  \includegraphics[page=2]{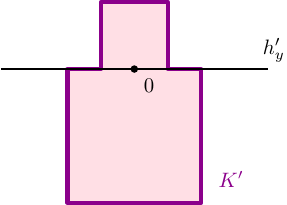}
\caption{A nonconvex polygon $K'$ and $y$ such that the total derivative of $f$
  does not exist at $y$. }
  \label{f:two_squares}
\end{center}
\end{figure}
\section{One more critical point}
\label{s:one_more_point}
%
In this section, we prove Proposition~\ref{p:one_more_point}. 
%
%
Given a manifold $M$ and a continuous function $f\colon M \to \R$ and $s \in
\R$ we define the \emph{level set} $L_s := \{w \in M\colon  f(w) = s\}$. 
In the proof of Proposition~\ref{p:one_more_point} we will need that the
level sets are well-behaved in the neighborhoods of points $u$ for which the
total derivative $Df(u)$ is nonzero.

%

\begin{proposition}
\label{p:path_connected}
Let $n\geq 1$, $f\colon \R^n \to \R$ be a $C^1$ function and $u \in \R^n$ be
  such that $Df(u) \neq 0$. Then there is a neighborhood $N(u)$ of $u$ such that 
  for every $v,w \in N(u)$ if $f(v) = f(w)$, then $v$ and $w$ can be connected
  with a path within the level set $L_{f(v)}$. (It is allowed that this path
  leaves $N(u)$ provided that it stays in $L_{f(v)}$.)
\end{proposition}

\begin{proof}
Without loss of generality assume
that $\frac{\partial f}{\partial x_n}(u) > 0$, otherwise we permute the
coordinates and/or swap $x_n$ and $-x_n$. Consistently with the previous
  section, given $x \in \R^n$, we write $x =
(\hat x, x_n)$ where $\hat x \in \R^{n-1}$ and $x_n \in \R$. Now we
consider the $C^1$ function $F \colon \R^{n-1} \times \R \times \R \to \R$ defined as
$F(\hat x, t, x_n) := f(\hat x, x_n) - t$. Note that
$\frac{\partial F}{\partial x_n} = \frac{\partial f}{\partial x_n}$. We also
observe that $F(\hat u, f(u), u_n) = 0$. Therefore, by the implicit
function theorem, there is an open neighborhood $N'$ of $(\hat u,
f(u))$
in $\R^{n-1} \times \R$ such that there is a $C^1$ function $g \colon N' \to
\R$ with $g(\hat u, f(u)) = u_n$ and that $F(\hat v, t,
g(\hat v, t)) = 0$ for any $(\hat v, t) \in
N'$. From the definition of $F$ this gives
\begin{equation}
  \label{e:fgt}
  f(\hat v, g(\hat v, t)) = t.
\end{equation}

By possibly restricting the neighborhood to a smaller set, we can assume
that $N'$ is the Cartesian product of a neighborhood $N'(\hat u)$ of
$\hat u$ in $\R^{n-1}$ and $N'(f(u))$ of $f(u)$ in $\R$, and that both
$N'(\hat u)$ and $N'(f(u))$ are open balls. 
Moreover, 
we can assume
that $\frac{\partial F}{\partial x_n}(\hat v, t, v_n) > 0$ for any
$(\hat v, t, v_n) \in N' \times N''(u_n)$ where $N''(u_n)$ is some
neighborhood of $u_n$ in $\R$, again a ball. Now we possibly further restrict
  $N'(\hat
u)$ and $N'(f(u))$ so that $g(\hat v, t)$ belongs to $N''(u_n)$ for any
$(\hat v, t) \in N'$.


The condition on the partial derivative of $F$
implies that for every $(\hat v, t) \in N'$ the equation $F(\hat v,
t, x_n) =
0$ has at most one solution $x_n \in N''(u_n)$. Therefore it has a unique
solution $x_n = g(\hat v,
t)$. In other words we get:
\begin{equation}
  \label{e:if_then}
  \hbox{If } f(\hat v, x_n) = t, \hbox{ then } x_n = g(\hat v,t).
\end{equation}

Now, we define $N(u) := \Psi^{-1}(N')$ where $\Psi\colon \R^{n-1} \times \R \to
\R^{n-1} \times \R$ is defined as $\Psi(v) = (\hat v, f(v))$ for any $v
\in \R^{n-1} \times \R$. 
In particular $(\hat v, f(v))$ belongs to $N'$ for any $v \in N(u)$.

  Let $t := f(v) = f(w)$. From~\eqref{e:if_then} we get $v_n = g(\hat v,
  t)$ and $w_n = g(\hat w, t)$. Let us consider an arbitrary path
  $P\colon [0,1] \to N'(\hat u)$ connecting $\hat v$ and $\hat w$. Let
  us `lift' $P$ to a path $P_t \colon [0,1] \to \R^{n-1} \times \R$ given by 
  $P_t(s)  := (P(s), g(P(s), t))$. This is a path connecting $v$ and $w$. We will
  be done once we show $P_t([0,1]) \subseteq L_t$. This means that we are
  supposed to show that $f(P(s), g(P(s), t)) = t$
  for
  every $s \in [0,1]$ which follows from~\eqref{e:fgt}. 
\end{proof}

\paragraph*{Total derivatives and gradients.}
Let $f \colon \R^n \to \R$ be a $C^1$ function. Then for any $u \in \R^n$, the
total derivative $Df(u)$ is represented by a row vector
$  \left( \frac{\partial}{\partial x_1}f(u), \dots, \frac{\partial}{\partial
  x_n}f(u)\right)$ (if $Df(u)$ exists). By $\|Df(u)\|$ we mean the Euclidean norm
  of this vector. The gradient $\nabla f(u)$ is the same vector transposed
\[
  \nabla f(u) := \left( \frac{\partial}{\partial x_1}f(u), \dots,
  \frac{\partial}{\partial
    x_n}f(u)\right)^T.
  \]
Then $\|\nabla f(u) \| = \|Df(u)\|$, and in addition $Df(u)(\nabla f(u)) =
\|Df(u)\|^2$.

Let $x \in \R^n$ and $\rho > 0$, by $B(x,\rho)
\subseteq \R^n$ we
denote the 
compact ball of radius $\rho$ centered in $x$ with respect to the
standard Euclidean metric.


\begin{lemma}
  \label{l:larger_f}
  Let $f \colon \R^n \to \R$ be a $C^1$ function, let $x \in \R^n$ and let
  $\zeta, \rho >
  0$. Assume that $\|Df(u)\| \geq \zeta$ for every $u \in B(x, \rho)$. Then
  there is $v \in B(x, \rho)$ such that $f(v) \geq f(x) + \frac{\zeta\rho}2$.
\end{lemma}

\begin{figure}
  \begin{center}
    \includegraphics{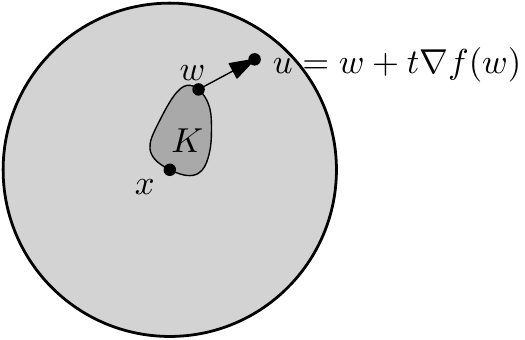}
    \caption{The set $K$ inside $B(x,\rho)$. For contradiction $M < f(x) +
    \frac{\rho\zeta}2$ which implies that $K$ does not touch the boundary of
    the ball. Then $t > 0$ can be chosen so that $u = w + t \nabla f(w)$ still
    belongs to $B(x,\rho)$.}
    \label{f:gradient}
  \end{center}
\end{figure}

\begin{proof}
Let 
  \[ K := \{y \in B(x, \rho)\colon f(y) \geq f(x) + \frac\zeta2\|y -
  x\|\}.
    \]
This is a closed therefore compact set, it is also nonempty because $x \in K$.
%
  Let $M := \max\{f(y)\colon y \in K\}$, for contradiction $M <
  f(x) + \frac{\zeta\rho}2$. Let $w \in K$ be such that $f(w) = M$. Note that
  for every $v \in \partial B(x, \rho) \cap K$ we get $f(v) \geq f(x) +
  \frac{\zeta\rho}2$ because $\|v-x\| = \rho$ in this case. Thus, in
  particular, $w \not \in \partial B(x, \rho)$. See Figure~\ref{f:gradient}.

  Consider the derivative at $w$ in the direction of the gradient $\nabla
  f(w)$.
  From properties of the total derivative, we get
  \[
  \lim_{t \to 0} \frac{|f(w + t \nabla f(w)) - f(w) -
  Df(w)(t\nabla f(w))|}{t\|\nabla f(w)\|} = 0.
  \]
 Therefore, for small enough $t > 0$ we get 
  \[|f(w + t \nabla f(w)) - f(w) -  t\|Df(w)\|^2| \leq \frac{\zeta}2t
  \|Df(w)\|.\] 
  Consequently,
  \[f(w) + t\|Df(w)\|^2 - f(w + t \nabla f(w)) \leq \frac{\zeta}2t
    \|Df(w)\|.\]
  Let $u = w + t \nabla f(w)$. If $t$ is small enough then $u \in B(x, \rho)$.
  Using $\|Df(w)\| \geq \zeta$ this gives
\begin{equation}
\label{e:f(u)nabla}
  f(u) \geq f(w) + t\|Df(w)\|^2 - \frac{\zeta}2t
      \|Df(w)\| \geq f(w) + \frac{\zeta}2t\|Df(w)\| = f(w) +
      \frac{\zeta}2t\|\nabla f(w)\|.
\end{equation}

  Because $w \in K$ and $t\|\nabla f(w)\| = \|u-w\|$, we further get
\begin{equation}
\label{e:f(u)norm}
  f(u) \geq f(x) + \frac{\zeta}2(\|w - x\| + \|u - w\|) \geq f(x) +
  \frac{\zeta}2(\|u-x\|).
\end{equation}
  Equation~\eqref{e:f(u)norm} gives that $u \in K$ while~\eqref{e:f(u)nabla}
  gives $f(u) > M$. This is a contradiction
with the choice of $M$.
\end{proof}

\begin{proof}[Proof of Proposition~\ref{p:one_more_point}]
 First, we can assume that all local extrema $m_1,
  \dots, m_k$ are strict. Indeed, if some of them is not strict, say $m_1$,
  then we can find $u \neq m_1, \dots, m_k$ with $Df(u) = 0$ in a neighborhood
  of $m_1$.

Next, because $k \geq 3$, there are at least two local maxima or two local
  minima among $m_1, \dots, m_k$. Without loss of generality, $m_1$ and $m_2$
  are local maxima. 
 
  Now, let us
  consider a path $\gamma \colon [0,1] \to S^n$ such that $\gamma(0) = m_1$ and
  $\gamma(1) = m_2$. Let $\min_f(\gamma):=
  \min\{f(\gamma(t)) \colon  t \in [0,1]\}$ (the minimum exists by compactness)
  and let $s := \sup(\min_f(\gamma))$
  where the supremum is taken over all $\gamma$ as above. 

 Before we proceed with the formal proof, let us sketch the main idea of the
  proof; see also Figure~\ref{f:mountains}. For contradiction assume that $Df(u) \neq 0$ for every $u \in
  S^n\setminus\{m_1, \dots, m_k\}$. Consider $\gamma$ such that $\min_f(\gamma)$ is very close to $s$. We
  will be able to argue that we can assume that such $\gamma$ is not close to
  any of the other extremes $m_3, \dots, m_k$. This guarantees that
  $\|Df(\gamma(t))\|$ is bounded from $0$ for every $t \in [0,1]$ except the
  cases when $\gamma(t)$ is close to $m_1$ or $m_2$. Using
  Lemma~\ref{l:larger_f}, we will be able to modify $\gamma$ to $\gamma'$ with 
 $\min_f(\gamma') > s$ obtaining a contradiction with the definition of $s$.

  \begin{figure}
    \begin{center}
      \includegraphics{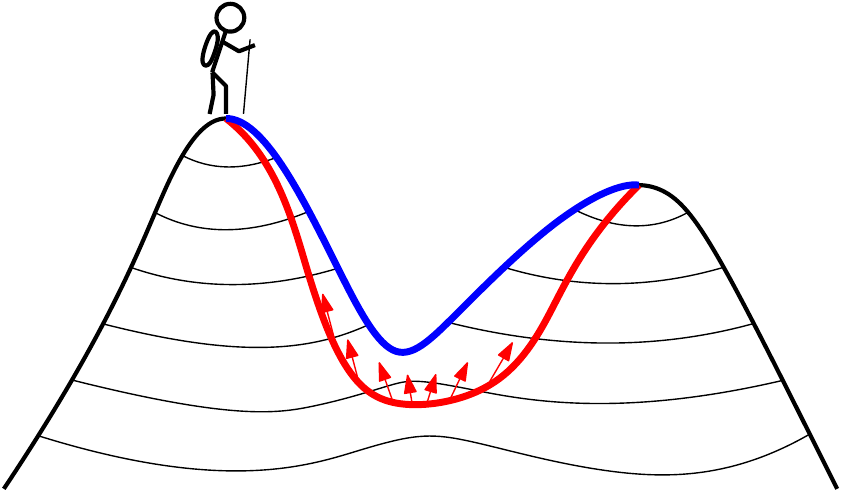}
      \caption{If we are in mountains and we want to hike from one peak to
      another without losing too much altitude, then the best way is to pass through a saddle point (see the upper path
      in blue). If we do not pass very close to a saddle point, then the positive
      gradient allows us to improve the path (see the lower path in red).}
      \label{f:mountains}
    \end{center}
  \end{figure}

 In further consideration, we consider the standard metric on $S^n$ obtained by
  the standard embedding of $S^n$ into $\R^{n+1}$ and restricting the Euclidean
  metric on $\R^{n+1}$ to a metric on $S^n$. For every $i \in [k]$, we pick 
  two closed metric\footnote{By a metric ball we mean a ball with a given
  center and radius. This way, we distinguish a metric ball from a general
  topological ball.} balls $B_i$ and $B'_i$ centered in $m_i$. Namely, $B_i$
  is chosen so that $m_i$ is a global extreme on $B_i$. We also assume that the
  balls $B_i$ are pairwise disjoint. Next, we distinguish whether $m_i$ is a
  local maximum or minimum. If $m_i$ is a local maximum, let us define $a_i :=
  \max\{f(x)\colon x  \in \partial B_i\}$. Note that $f(m_i) > a_i$ as $m_i$ is
  a global maximum on $B_i$. Then we pick a closed ball $B'_i$ centered in $m_i$
  inside $B_i$ so that $f(x) > a_i$ for every $x \in B'_i$. 
  If $m_i$ is a local
  minimum, we proceed analogously. We set $a_i :=
    \min\{f(x)\colon x  \in \partial B_i\}$ and we pick $B'_i$ so that $f(x)
    < a_i$ for every $x \in B'_i$. 
    For later use, we also define $a'_i := \min\{f(x)\colon x \in B'_i\}$ for
    $i \in \{1,2\}$. Note that $a'_i > a_i$. 
    
    Given a path $\gamma$ connecting
    $m_1$ and $m_2$, we say that $\gamma$ is \emph{avoiding} if it does not
    pass through the interior of any of the balls $B'_3, \dots, B'_k$.

  \begin{claim}
\label{c:avoiding}
    Let $\gamma$ be a path connecting $m_1$ and $m_2$. Then there is an avoiding
    path $\bar \gamma$ connecting $m_1$ and $m_2$ such that $\min_f(\bar \gamma) \geq
    \min_f(\gamma)$.
  \end{claim}

  \begin{proof}
  Assume that $\gamma$ enters a ball $B'_i$ for $i \in \{3, \dots, k\}$. Let us
    distinguish whether $m_i$ is a local maximum or minimum.

    First assume that $m_i$ is a local maximum. Then $\min_f(\gamma) \leq a_i$
    because $\gamma$ has to pass through $\partial B_i$. By a homotopy,
    fixed outside the interior of $B'_i$
    we can assume that $\gamma$ avoids $m_i$ (here we use
    $n \geq 2$); see, e.g., the proof of Proposition~1.14 in~\cite{hatcher02}
    how to perform this step.\footnote{We point out that the current online
    version of~\cite{hatcher02} contains a different proof of Proposition~1.14.
    Therefore, here we refer to the printed version of the book.} In addition, by further homotopy fixed outside the
    interior of $B'_i$ we can modify $\gamma$ so that it avoids the interior of
    $B'_i$ (the second homotopy pushes $\gamma$ in direction away from $m_i$).
    This does not affect $\min_f(\gamma)$ because $f(x) > a_i$ for every $x \in
    B'_i$. 

    Next let us assume that $m_i$ is a local minimum. Then $\min_f(\gamma) < a_i$
    because $\gamma$ has to pass through $\partial B'_i$ (this is not a
    symmetric argument when compared with the previous case). Modify $\gamma$
    by analogous homotopies as above; however, this time with respect to $B_i$
    (so that $\gamma$ completely avoids the interior of $B_i$). Because
    $\min_f(\gamma) < a_i$ and $f(x) \geq a_i$ for $x \in \partial B_i$, the
    minimum of $\gamma$ cannot decrease by these modifications.
    By performing these modifications for all $B'_i$ when necessary, we get the
    required $\bar \gamma$.
  \end{proof}

  
  Now, let us consider a diffeomorphism $\psi \colon S^n \setminus \{m_k\} \to
  \R^n$ given by the stereographic projection (in particular, it maps closed
  balls avoiding $m_k$ to closed balls). Let $g \colon \R^n
  \to \R$ be defined as $g := f\circ \psi^{-1}$. Let $n_i := \psi(m_i)$ for $i
  \in [k-1]$.  Once we find $v \in \R^n$, $v \neq n_1, \dots, n_{k-1}$ such
  that $Dg(v) = 0$, then $u := \psi^{-1}(v)$ is the required point with
  $Df(u)=0$.  Note that $n_1, n_2$ are still local maxima of $g$ and $n_3,
  \dots, n_{k-1}$ are local maxima or minima. We also set $D_i := \psi(B_i)$
  and $D'_i := \psi(B'_i)$ for $i \in [k-1]$ and $C_k := \psi(B_k \setminus
  \{m_k\})$, $C'_k := \psi(B'_k \setminus \{m_k\})$. The sets $D_i$ and $D'_i$
  are closed (metric) balls centered in $n_i$ whereas $C_k$ and $C'_k$ are
  complements of open (metric) balls in $\R^n$. Let $K$ be the compact set
  obtained from $\R^n$ by removing the interiors of $D'_1, \dots, D'_{k-1},
  C'_k$. Let us fix small enough $\eta > 0$ such that the closed
  $\eta$-neighborhoood $K_\eta$ of $K$ avoids $n_1, \dots, n_{k-1}$. We will
  also use the notation $K_{\eta/3}$ for the closed $\frac \eta3$-neighborhood
  of $K$. See Figure~\ref{f:Ks}.

\begin{figure}
  \begin{center}
    \includegraphics{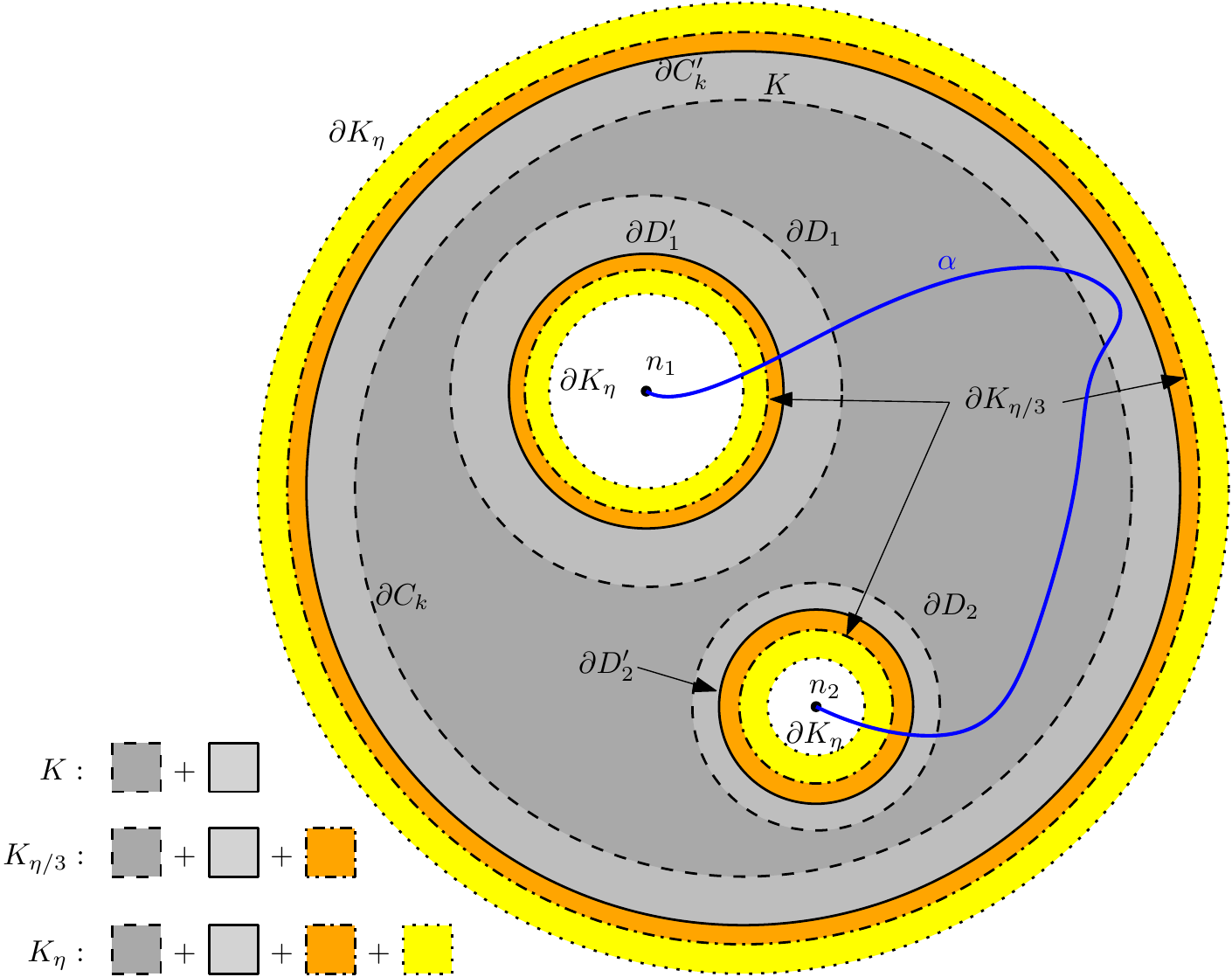}
    \caption{The sets $K$, $K_{\eta/3}$ and $K_\eta$ and some path $\alpha$
    connecting $n_1$ and $n_2$ of the form $\alpha = \psi \circ \gamma$ where
    $\gamma$ is avoiding. In the picture, $k = 3$.}
    \label{f:Ks}
  \end{center}
\end{figure}

  Assume, for contradiction, that $K_\eta$ does not contain $v$ with $Dg(v) =
  0$. Because $K_\eta$ is compact and $g$ is $C^1$, there is $\zeta > 0$ such
  that $\|Dg(w)\| \geq \zeta$ for every $w \in K_\eta$. 

  For every $w \in
K_{\eta/3}$ let $N(w)$ be the neighborhood given by Proposition~\ref{p:path_connected} (the neighborhood is
considered in the whole $\R^n$ not only in $K_{\eta/3}$). 
By possibly restricting $N(w)$ to smaller sets,
we can assume that each $N(w)$ is open and fits into a ball of radius $\frac23
\eta$. (In particular, if $w \in K_{\eta/3}$, then $N(w) \subseteq K_\eta$.)

\begin{claim}
\label{c:Lebesgue}
  There is $\varepsilon > 0$ such that for every $x \in K_{\eta/3}$ the metric ball
  $B(x,\varepsilon) \subseteq \R^n$ centered in $x$ of radius $\varepsilon$ fits into $N(w)$ for
  some $w \in K_{\eta/3}$.
\end{claim}

\begin{proof}
This is just a modification of the Lebesgue number lemma.
  Let us consider the
  open cover $\O$ of $K_\eta$ consisting of all sets $N(w)$ together with the
  relative interiors of the sets $B'_1 \cap (K_\eta \setminus K_{\eta/3}),
  \dots, B'_{k-1} \cap (K_\eta \setminus K_{\eta/3}),
  C'_k \cap (K_\eta \setminus K_{\eta/3})$ (all sets are relatively open in $K_\eta$). Note that the newly added sets are
  disjoint from $K_{\eta/3}$. Let $\varepsilon > 0$ be the standard Lebesgue number with
  respect to the cover $\O$, that is, for every $x \in K_\eta$, the ball
  $B(x,\varepsilon)$ fits into one of the sets of $\O$; see~\cite[Lemma~27.5]{munkres}. Then the required claim
  holds with this $\varepsilon$ because if $x \in K_{\eta/3}$, then $x$ does not belong to
  any of the newly added sets of $\O$.
\end{proof}

Let $\varepsilon$ be the value obtained from Claim~\ref{c:Lebesgue}. Because
some ball $B(x, \varepsilon)$ fits into some $N(w)$ which fits into a ball of
radius $\frac23 \eta$, we get $\varepsilon \leq \frac 23\eta$.

Let us consider a path $\gamma$ in $S^n$ such that
\begin{enumerate}[(s1)]
  \item $s - \min_f(\gamma) < a'_1 - a_1$;
  \item $s - \min_f(\gamma) < a'_2 - a_2$; and
  \item $s - \min_f(\gamma) < \frac{\zeta\varepsilon}4$.
\end{enumerate}
By Claim~\ref{c:avoiding}, we can assume that
$\gamma$ is avoiding.
We will start modifying $\gamma$ to $\gamma'$ with
$\min_f(\gamma') > s$, which will be the required contradiction. 
Let $\alpha :=
\psi \circ \gamma$; see the diagram at Figure~\ref{f:diag}. Then $\alpha$ connects $n_1$ and $n_2$, and $\alpha$ avoids
the interiors of $D'_3, \dots, D'_{k-1}$ and $C'_k$; see Figure~\ref{f:Ks}. 

\begin{figure}
\begin{center}
\begin{tikzcd}
 \left[0,1\right] \arrow[r,"\gamma"]\arrow[dr, "\alpha"'] & S^n\setminus\{m_k\} \subseteq S^n \arrow[r, "f"] \arrow[d,"\psi"] & \R \\
 & \R^n \arrow[ur, "g"']&
\end{tikzcd}
\end{center}
\caption{The maps $\alpha$, $\gamma$, $\psi$, $f$ and $g$. The two triangles
  are commutative.}
\label{f:diag}
\end{figure}
Because, $\alpha$ is a continuous function on the compact interval $[0,1]$, we
get, by the Heine-Cantor theorem, that $\alpha$ is uniformly continuous. In particular, there is $\delta > 0$
such that if $t_1, t_2 \in [0,1]$ with $|t_1 - t_2| \leq \delta$, then
$\|\alpha(t_1) - \alpha(t_2)\| \leq \frac{\varepsilon}3$. Let us consider a
positive integer $\ell > \frac 1\delta$. We will be modifying $\alpha$ in two
steps. First, we get $\alpha''$ such that $\alpha''(t) > s$ if $t =
\frac{j}\ell$ for some $j \in \{0, \dots, \ell\}$. Then we modify $\alpha''$
individually on the intervals $(\frac j \ell, \frac{j+1}\ell)$ for $j \in \{0,
\dots, \ell-1\}$ obtaining $\alpha'$ with $\min_g(\alpha') > s$. (Given a path $\beta \colon [0,1] \to \R^n$ connecting $n_1$
and $n_2$, we define $\min_g(\beta) := \min\{g(\beta(t))\colon t \in [0,1]\} =
\min_f(\psi^{-1}\circ \beta)$.)
The required $\gamma'$ will be obtained as $\psi^{-1} \circ \alpha'$. 

For the first step, let us first say that an interval $I_j = [\frac j \ell,
\frac{j+1}\ell]$ where $j \in \{0, \dots, \ell-1\}$ \emph{requires a
modification} if $g(\alpha(t)) \leq s$ for some $t \in I_j$. This in
particular
means that $\alpha(t) \in K$ for this $t$: Indeed, this follows from (s1) and
(s2). We already know that $\alpha$ avoids the interiors of $D'_3, \dots,
D'_{k-1}$ and $C'_k$. It remains to check that $\alpha(t)$ does not belong to
the interiors of $D'_1$ and $D'_2$ as well. Because $\alpha$ has to meet
$\partial D_1$ and $\partial D_2$, we get that $\min_f(\gamma) = \min_g(\alpha)
\leq a_1, a_2$ from the definition of $a_1$ and $a_2$. By (s1) and (s2), we get
$s < a'_1, a'_2$. Therefore, from the definition of $a'_1$ and $a'_2$, we get
that $\alpha(t)$ cannot belong neither to $D'_1$ nor to $D'_2$ as required.

By the uniform continuity, the fact that $g(\alpha(t)) \leq s$ for some $t \in
I_j$ implies that $\alpha(I_j)$ belongs to the closed
$\frac{\varepsilon}3$-neighborhood of $K$. In particular, $\alpha(I_j)$ belongs
to $K_{\eta/3}$ as $\varepsilon \leq \frac23\eta < \eta$.

Now, for each $I_j$ which requires a modification, consider the open
$\varepsilon$-ball $U_j \subseteq \R^n$ centered in $\alpha(\frac{2j+1}{2\ell})$.
(Note that, $\frac{2j+1}{2\ell}$ is the midpoint of $I_j$.) From the previous
considerations, the centre of each $U_j$ belongs to $K_{\eta/3}$ and the whole
$U_j$ is a subset of $K_\eta$.

Now we perform the first step. Consider $t = \frac j\ell$ for some $j \in
\{0,\dots, \ell\}$. If $g(\alpha(t)) > s$, then we do nothing. Note that this
includes the cases $j = 0$ or $j=\ell$. If $g(\alpha(t)) \leq
s$, then both intervals $I_{j-1}$ and $I_j$ require a modification. 
By the uniform continuity, the open ball $V_j
\subseteq \R^n$
centered in $\alpha(t)$ of radius $\frac{2\varepsilon}3$ is a subset of both
$U_{j-1}$ and $U_j$; see Figure~\ref{f:UV}. We observe that $V_j$ is a subset of $K_\eta$ as $V_j
\subseteq U_j$.
In particular, by the
definition of $\zeta$, we get that $\|Dg(w)\| \geq \zeta$ for every $w \in
V_j$. By Lemma~\ref{l:larger_f}, used on a closed ball of a slightly smaller
radius $\frac{\varepsilon}2$, there is a point $v$ in $V_j$ such that 
\[g(v) \geq
g(\alpha(t)) + \frac{\zeta\varepsilon}4 \geq {\textstyle \min_g(\alpha)} +
\frac{\zeta\varepsilon}4 = {\textstyle \min_f(\gamma)} +
\frac{\zeta\varepsilon}4.\]
Using (s3), we get $g(v) > s$. Now, by a homotopy, we modify $\alpha$ to
$\alpha''$ so that
it stays fixed outside the interval $(t - \frac1{4\ell}, t+\frac1{4\ell})$, the
modification of $\alpha$ occurs only in $V_j$ and $\alpha''(t) = v$; see
Figure~\ref{f:12step}.
We perform
these modifications simultaneously for every $t = \frac j\ell$ with
$g(\alpha(t)) \leq 
s$. This is possible as the intervals $[t - \frac1{4\ell}, t+\frac1{4\ell}]$ are
pairwise disjoint. This way, we obtain the required $\alpha''$.

\begin{figure}
 \begin{center}
   \includegraphics[page=1, scale=.8]{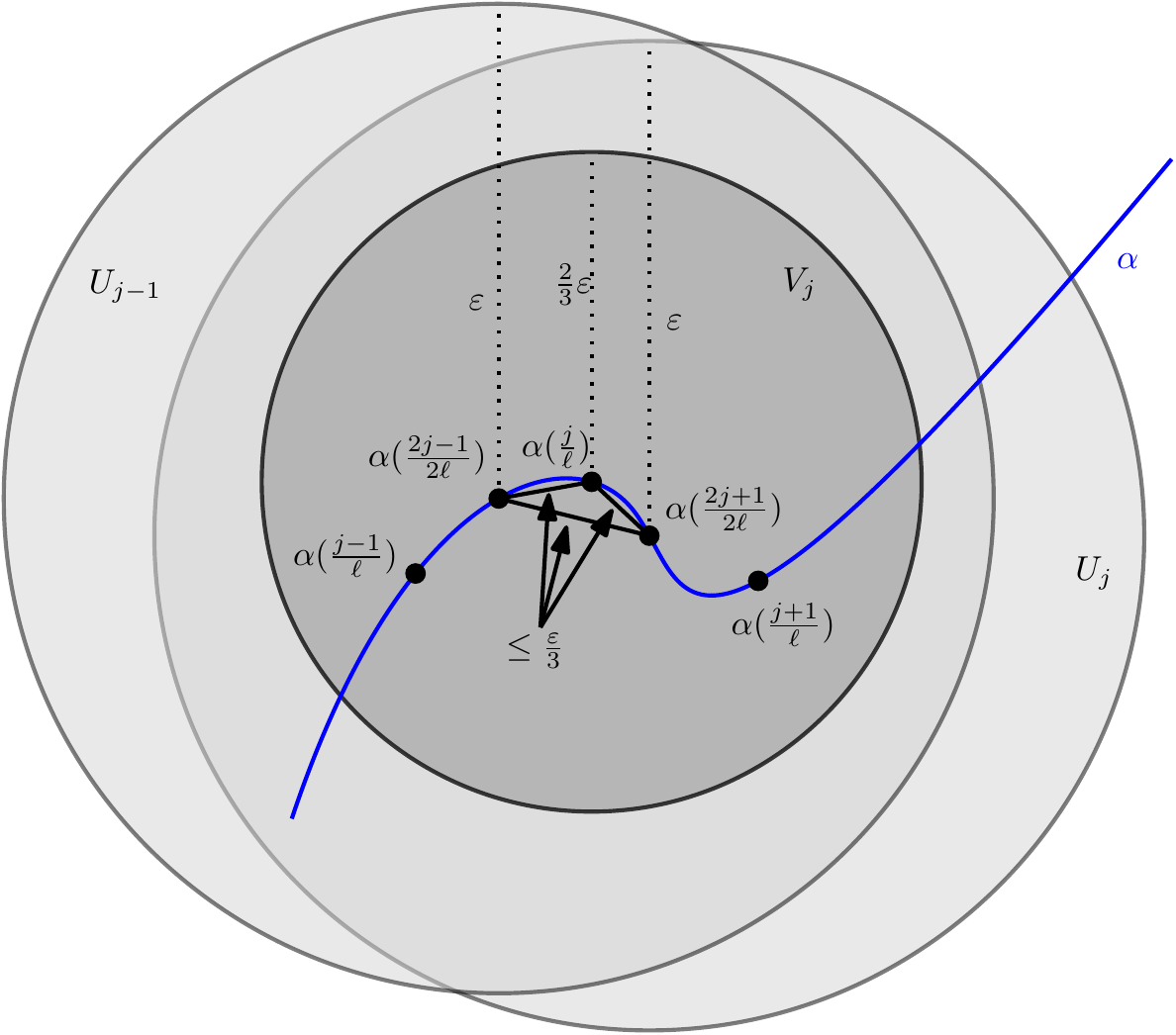}
   \caption{The sets $U_{j-1}$, $U_j$ and $V_j$ in the case that
   $g(\alpha(\frac j\ell)) \leq s$.}
   \label{f:UV}
 \end{center}
\end{figure}

\begin{figure}
 \begin{center}
   \includegraphics[page=2, scale=.8]{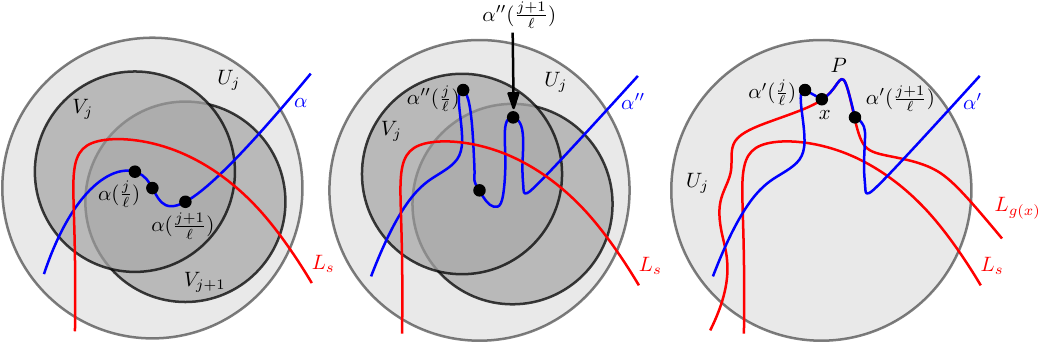}
   \caption{The first and the second step of modifications of $\alpha$ on an
   interval $I_j$ requiring a modification (the modification is shown only on
   this interval).}
   \label{f:12step}
 \end{center}
\end{figure}

Finally, we perform the second step of the modification. Let $I_j = [\frac
j\ell, \frac{j+1}\ell]$ be an
interval requiring a modification. We already know that
$g(\alpha''(\frac{j}{\ell})) > s$ and $g(\alpha''(\frac{j+1}{\ell})) > s$. In
addition, we know that both $\alpha''(\frac{j}{\ell})$ and
$\alpha''(\frac{j+1}{\ell})$ belong to $U_j$ as they belong to $V_j$ or
$V_{j+1}$. We set $\alpha'(\frac{j}{\ell}) := \alpha''(\frac{j}{\ell})$ and
$\alpha'(\frac{j+1}{\ell}) := \alpha''(\frac{j+1}{\ell})$. Next, we aim to
define $\alpha'$ on $(\frac{j}{\ell}, \frac{j+1}{\ell})$, which is the interior
of $I_j$, so that $\min(g(\alpha'(I_j))) > s$. By Claim~\ref{c:Lebesgue}, $U_j$
fits into some $N(w)$ for some $w \in K_{\eta/3}$. (Here we use that the center
of $U_j$ belongs to $K_{\eta/3}$.) Now, Proposition~\ref{p:path_connected} implies that
$\alpha'(\frac{j}{\ell})$ and $\alpha'(\frac{j+1}{\ell})$ may be connected by
a path $P\colon[0,1] \to \R^n$ such that $g(P(t)) > s$ for every $t \in [0,1]$:
Indeed, let us assume that, without loss of generality,
$g(\alpha'(\frac{j}{\ell})) \geq g(\alpha'(\frac{j+1}{\ell})) > s$. First, draw
$P$ as a straight line from $\alpha'(\frac{j}{\ell})$ towards $\alpha'(\frac{j+1}{\ell})$ until
we reach a (first) point $x \in U_j \subseteq N(w)$ with $g(x) =
g(\alpha'(\frac{j+1}{\ell}))$; of course, it may happen that $x =
\alpha'(\frac{j+1}{\ell})$. Then by Proposition~\ref{p:path_connected}, $x$ and
$\alpha'(\frac{j}{\ell})$ can be connected within the level set $L_{g(x)}$;
see Figure~\ref{f:12step}.
(This may mean that $P$ leaves $N(w)$, or even $K_{\eta}$, but this is not problem for the
argument.)
Altogether, we set $\alpha'$ on $I_j$ so that it follows the path $P$, and this
we do independently on each interval requiring a modification. Other intervals
remain unmodified.

From the construction, we get $\min_g(\alpha') > s$; therefore the path $\gamma'
:= \psi^{-1} \circ \alpha'$ satisfies $\min_f(\gamma') = \min_g(\alpha') > s$ which
contradicts the definition of $s$.
\end{proof}

\section{Depth-like functions with few critical points}
\label{s:conclusions}

\subparagraph*{Bipyramid over a triangle.} In $\R^3$, we have a candidate example
of a convex body, namely the regular bipyramid $B$ over an equilateral
triangle $T$, such that there are exactly four barycentric hyperplanes (with
respect to the barycenter of $B$, which coincides with the point of maximal
depth in this case).
On the one hand, this is not surprising, because this is $n + 1$
hyperplanes, where $n = 3$ is the dimension of the ambient space. On the other
hand, if this is true, then it answers negatively, in dimension $3$, a question 
from~\cite[A8]{croft_falconer_guy94}, whether $2^n - 1$ barycentric hyperplanes
always exist.

More concretely, we conjecture that the only barycentric hyperplanes are the
following: three planes perpendicular to $T$ which meet $T$ in lines realizing
the depth of $T$ (these would be the hyperplanes realizing the depth), and the
plane of $T$ (this is the one extra plane).  Unfortunately, in this case, it is
not so easy to analyze the depth function as in the case of $T \times I$.

\subparagraph*{A function with four critical points and many properties of the depth.}
Let us recall that the depth function $\delta \colon S^{n-1} \to [0,1]$ 
on a convex body satisfies the following properties: 

\begin{enumerate}[(i)]
 \item \label{e:i} $\delta(v) = 1 - \delta(-v)$;
 \item \label{e:ii} $0 \in \conv U$ where $U \subseteq S^{n-1}$ is the set of
   the points where $\delta$ attains the minimum (by Proposition~\ref{p:irbt});
 \item \label{e:iii} $|U| \geq 3$, (by Corollary~\ref{c:three});
 \item \label{e:iv} $\delta$ is $C^1$ (by Proposition~\ref{p:C1});
 \item \label{e:v} if $U$ is finite, then $\delta$ has at least one more pair of opposite critical points (by Proposition~\ref{p:one_more_point}
and by (\ref{e:i})).
\end{enumerate}

We will show that our argument in the proof of Theorem~\ref{t:four} is tight in
the sense that for $n \geq 3$ there exists a function $\delta' \colon S^{n-1} \to [0,1]$
satisfying (\ref{e:i})--(\ref{e:v}) with equalities in (\ref{e:iii}) and (\ref{e:v}). 
%

In order to define $\delta'$, it will be much more convenient to reparametrize
$\delta'$. Thus, we will exhibit $\delta'' \colon S^{n-1} \to [-1, 1]$ which
satisfies (\ref{e:ii})--(\ref{e:v}) with equalities in (\ref{e:iii}) and (\ref{e:v}) but $\delta''(v) =
-\delta''(-v)$ instead of (\ref{e:i}). Then the required $\delta'$ is obtained as
$\frac12 \delta'' + \frac 12$.

This time we decompose $\R^{n}$ as $\R^{n-2}
\times \R^2$ and for a point $x \in \R^{n}$ we write $x = (y; z)$ where $y =
(y_1, \dots, y_{n-2}) \in
\R^{n-2}$ and $z = (z_1, z_2) \in \R^2$. The idea is to define $\delta''$
separately on the sphere $S^{n-3} \times \{0\}$ so that there is only one pair
of opposite critical points here (this will be the extra pair from (\ref{e:v})),
separately on the sphere $\{0\} \times S^1$ so that there are three pairs of
critical points (these will be three global minima and three global maxima from
(\ref{e:iii})), and then merge the two constructions so that the resulting function is
smooth and no new critical points arise. Unfortunately, the details are
somewhat tedious.

We actually define $\delta''$ on $\R^{n} \setminus
\{0\}$ considering $S^{n-1}$ as a subset of $\R^{n} \setminus
\{0\}$. 
Now, we set
\begin{equation}
\label{e:delta''first}
  \delta''(y,z) = \frac1{10}(2\|y\| - \|y\|^3)y_1 + \frac12\|z\|(z_1^3 - z_1z_2^2 - 2z_1z_2^2).
\end{equation}
We remark that the expression $(z_1^3 - z_1z_2^2 - 2z_1z_2^2)$ is nothing else
then the real part $\Re(z^3)$, where $z = (z_1, z_2)$ is identified with the
complex number $z_1 + iz_2$. From~\eqref{e:delta''first} we easily see that
$\delta''$ is smooth on  $\R^{n} \setminus \{0\}$, therefore its restriction to $S^{n-1}$ is smooth as well as
the inclusion $S^{n-1} \subseteq \R^{n} \setminus \{0\}$ is a smooth embedding. 
We also easily check that $\delta''(y,z) = - \delta''(-y, -z)$.

From now on, let us assume that $(y,z) \in S^{n-1}$, that is $\|y\|^2 + \|z\|^2 = 1$.
If $y \neq 0$, we get
\begin{equation}
\label{e:delta''y}
  (2\|y\| - \|y\|^3)y_1 = (1 - (1 - \|y\|^2)^2)\frac{y_1}{\|y\|} =
  (1-\|z\|^4)\frac{y_1}{\|y\|},
\end{equation}
and if $z \neq 0$, we get (in complex numbers)
\begin{equation}
\label{e:delta''z}
  \|z\|(z_1^3 - z_1z_2^2 - 2z_1z_2^2) = \|z\|^4 \frac{\Re (z^3)}{\|z\|^3} =
  \|z\|^4 \Re((z/\|z\|)^3).
\end{equation}
Altogether~\eqref{e:delta''first}, \eqref{e:delta''y} and~\eqref{e:delta''z}
give
\begin{equation}
\label{e:delta''second}
  \delta''(y,z) = 
  \begin{cases}
    (1-\|z\|^4)\frac1{10}\frac{y_1}{\|y\|} + \|z\|^4 \frac12\Re((z/\|z\|)^3)
 \hbox{ if } y,z \neq 0,\\
     \frac12\Re(z^3) 
 \hbox{ if } y = 0,\\
    \frac1{10}y_1
\hbox{ if } z = 0.\\   
  \end{cases}
\end{equation}

In particular,~\eqref{e:delta''second} implies that for $y,z \neq 0$, $\delta''(y,z)$ is a convex combination of $\frac1{10}
\frac{y_1}{\|y\|}$ and $\frac12\Re((z/\|z\|)^3)$, which attain values in
$[-1/10, 1/10]$ and $[-1/2,1/2]$ respectively.  Therefore
$\delta''(S^{n-1}) \subseteq [-1/2,1/2]$.

Now we check that $\delta''$ attains exactly three global minima on $S^{n-1}$.
We observe that $\delta''(0,e^{i(2k+1)/3)\pi}) = -1/2$ for $k = 0,1,2$
by~\eqref{e:delta''second}. Therefore $\delta''$ attains the
minimum at these three points. On the other hand, we realize that these are the
only three points where $\delta''(y,z) = -1/2$. Indeed, if $y = 0$, then
$\delta''(y,z) = -1/2$ only if $\Re(z^3) = -1$, which occurs only if
$z = e^{i(2k+1)\pi/3}$ for $k = 0,1,2$. If $z = 0$, then $\delta''(y,z) \geq
-1/10$ by~\eqref{e:delta''second}. Finally, if $y,z \neq
0$, then the convex combination from~\eqref{e:delta''second} 
has the strictly
positive coefficient $(1- \|z\|^4)$ at $\frac1{10} \frac{y_1}{\|y\|}$, which
implies $\delta''(y,z) > \frac12\Re((z/\|z\|)^3) \geq -1/2$. This
characterization of global minima also gives property (\ref{e:ii}).

It remains to check that there is exactly one extra pair of opposite critical
points of $\delta''$. This could be done via Lagrange multipliers but the
computations seem to be slightly tedious, thus we provide a different argument.
In advance, we announce that these extra critical points
will be $(e_1, 0)$ and $(-e_1,0)$, where $e_1 \in S^{n-3} \subseteq \R^{n-2}$ is the first coordinate
vector $e_1 = (1,0,\dots, 0)$. We will rule out all other options, thus
these points have to be indeed critical by Proposition~\ref{p:one_more_point}.

Let $(y,z) \in S^{n-1}$ be a critical point. If $y = 0$, then $\delta''(0, z) =
\frac 12\Re(z^3)$ by~\eqref{e:delta''second} when restricted to $\{0\} \times
S^1 \subseteq S^{n-1}$ (where $0 \in \R^{n-2}$ in this case). Therefore $(0,
z)$ has to be critical point of the restriction as well. It is easy to analyse
that the only critical points of $\frac 12\Re(z^3)$ are of the form $z =
e^{ik\pi/3}$ where $k \in \{0,\dots, 5\}$, which are the minima and the maxima
opposite to the minima. If $z = 0$, then $\delta''(y,0) = \frac{1}{10}y_1$
when restricted to $S^{n-3} \times \{0\} \subseteq S^{n-1}$ (where $0 \in
\R^2$ in this case). Again, it is easy to analyze that $e_1$ and $-e_1$ are the
only critical points. (Here they are the maximum and minimum in the restriction
respectively, but they are not even local extremes on whole $S^{n-1}$.)
Finally, we consider the case $y, z \neq 0$. First, we fix $y$ and let $z$ vary
subject to $\|y\|^2 + \|z\|^2 = 1$, which implies that $\|z\|$ is fixed as well. Then $\delta''(y, z) =  a_y + b_y\Re((z/\|z\|)^3)$ by~\eqref{e:delta''second} where
$a_y$ and $b_y$ are constants depending on $y$. This implies that if $(y,z)$ is
critical, then $z/\|z\| = e^{ik\pi/3}$ where $k \in \{0,\dots, 5\}$. Next, we fix
$z$ and let $y$ vary. By a similar idea as above, we deduce that if $(y,z)$ is
critical, then $y = (\pm y_1, 0,\dots,0)$. Finally, let us fix both $y$ and $z$
and consider the
$2$-plane $\rho(y,z)$ given by all vectors $(t_y y, t_z z)$ for $t_y, t_z \in
\R$. This $2$-plane meets $S^{n-1}$ in a circle. For $(t_y y, t_z z)$ in
$\rho(y,z) \cap S^{n-1}$ the equation~\eqref{e:delta''second} gives
\begin{align*}
    \delta''(t_y, t_z) &= (1 - t_z^4 \|z\|^{4})\frac{1}{10}\frac{y_1}{\|y\|} +
  t_z^4\|z\|\frac12 \Re((z/\|z\|)^3)\\ 
  &= \frac{1}{10}\frac{y_1}{\|y\|} + t_z^4
  \|z\|^{4}\left(\frac 12 \Re((z/\|z\|)^3) -
  \frac{1}{10}\frac{y_1}{\|y\|}\right).
\end{align*}
Therefore $(t_y y, t_z z)$, for $t_y, t_z \neq 0$, may be the critical point of
$\delta''$ only if 
\begin{equation}
\label{e:z_not_y}
  \frac 12 \Re((z/\|z\|)^3) = \frac{1}{10}\frac{y_1}{\|y\|}
\end{equation}
 which is independent of the values $t_y$ and $t_z$. If we recall the previous
 two conditions on the critical point $(y,z)$, we get $\Re((z/\|z\|)^3) = \pm
 1$ and $\frac{y_1}{\|y\|} = \pm1$, therefore~\eqref{e:z_not_y} may not hold
 simultaneously. This finishes the analysis of the critical points of
 $\delta''$.

\iffull
\section*{Acknowledgements} We thank Stanislav Nagy for introducing us to Gr\"{u}nbaum's questions,
for useful discussions on the topic, for providing us with many references, and
for comments on a preliminary version of this paper. We thank Jan Kyn\v{c}l and
Pavel Valtr for letting us know about a more general counterexample they found.
We thank Roman Karasev for providing us with
references~\cite{karasev11,blagojevic-karasev16} and for comments on a
preliminary version of this paper. Finally, we thank an anonymous referee for
many comments on a preliminary version of the paper which, in particular,
yielded an important correction in Section~\ref{s:one_more_point}.
\fi



\bibliography{depth}

\end{document}